\documentclass[11pt]{article}

\usepackage{amsmath,amssymb,amsthm,amscd}
\numberwithin{equation}{section} \textwidth 160mm \oddsidemargin
-1pt 

\def\p{\partial}
\def\b{\bar}

\def\tr{\rm tr}

\newtheorem{prop}{Proposition}[section]
\newtheorem{theo}[prop]{Theorem}
\newtheorem{lem}[prop]{Lemma}

\newtheorem{rem}[prop]{Remark}

\newtheorem{defi}[prop]{Definition}

\def\begeq{\begin{equation}}
\def\endeq{\end{equation}}

\def\and{\quad{\rm and}\quad}

\let\lra=\longrightarrow

\def\mapright\#1{\,\smash{\mathop{\lra}\limits^{\#1}}\,}

\title{ K\"ahler-Einstein metrics on Fano manifolds, I: approximation of metrics with cone singularities}

\begin{document}
\bibliographystyle{plain}
\date{}
\author{Xiu-Xiong Chen, Simon Donaldson and Song Sun}
\date{\today}
\maketitle

\section{Introduction}

This is the first of a series of three papers which provide proofs of results
announced in \cite{CDS}.
Let $X$ be a Fano manifold of complex dimension $n$. Let $\lambda>0$ be an integer and $D$ be a smooth divisor in the linear system $\vert -\lambda K_{X}\vert$. For $\beta\in (0,1]$ there is now a well-established notion of a K\"ahler-Einstein metric with a cone singularity of cone angle $2\pi \beta$ along $D$. (It is often called an \lq\lq edge-cone" singularity).  For brevity we will just say that $\omega$ has cone angle $2\pi \beta$ along $D$.   The Ricci curvature of such a metric $\omega$ is $(1-\lambda(1-\beta))\omega$. 
Our primary concern is the case of positive Ricci curvature, so we suppose throughout most of the paper that $\beta\geq \beta_{0}> 1-\lambda^{-1}$. However our arguments also apply to the case of non-positive Ricci curvature: see Remark \ref{remark neg}.

\begin{theo}  
If $\omega$ is a K\"ahler-Einstein metric with cone angle $2\pi \beta$ along $D$ with $\beta\geq \beta_0$,  then $(X,\omega)$ is the Gromov-Hausdorff limit of a sequence of smooth K\"ahler metrics with positive Ricci curvature and with diameter bounded by a fixed number depending only on $\beta_{0},\lambda$. \label{thm1}
\end{theo}

This suffices for most of our applications but we also prove a sharper statement.
\begin{theo}
If $\omega$ is a K\"ahler-Einstein metric with cone angle $2\pi \beta$ along $D$ then $(X,\omega)$ is the Gromov-Hausdorff limit of a sequence of smooth K\"ahler metrics $\omega_{i}$ with  
$ Ric(\omega_{i})\geq (1-\lambda(1-\beta))\omega_{i}$.\label{thm2}
\end{theo}

One consequence of our approximation results is a uniform bound on the Sobolev constant; this bound has also been obtained by Jeffres, Mazzeo and Rubinstein in \cite{JMR}.

A K\"ahler-Einstein metric with cone angle $2\pi \beta > 0$ along the divisor $D$,   satisfies
the equation of  currents:
\begin{equation}
Ric(\omega) = (1-(1-\beta)\lambda) \omega + 2\pi (1-\beta) [D], 
\end{equation}
where $[D]$ is the current of integration along $D$.
To prove our theorems, we first approximate $[D]$ by a sequence of smooth positive forms,
and solve the corresponding complex Monge-Amp\`ere equations; then we show that this sequence of solutions converges to the initial K\"ahler-Einstein metric as expected.  We will make this more precise in Section \ref{sec 2}. \\

 We will treat the case when $\lambda = 1$. The general case can be done in an identical way. \\

In this article we fix $\omega_0$ to be a smooth K\"ahler form in $2\pi c_1(X)$. Set the space of smooth K\"ahler potentials to be 
\[
{\cal H} = \{\varphi \in C^\infty(X;\mathbb R): \omega_0 + \sqrt{-1}\p \bar \partial \varphi > 0 \;\;{\rm in}\;\; X\}.
\]

\begin{defi} A  K\"ahler metric $\omega'$ on $X$ with cone angle $2\pi\beta$ along $D$ is a current in $2\pi c_1(X)$ such that 
\begin{enumerate}
\item $\omega'$ is a closed positive $(1,1)$ current on $X$, and is a smooth K\"ahler metric in  $X\setminus D;$
\item for any point $p \in D,\;$ there exists a chart $(\mathcal U, \{z_i\})$  
so that $z_1$ is a local defining function for  $D$ and on this chart the metric is uniformly equivalent to the standard cone metric: 
\[
\sqrt{-1}\sum_{j=2}^n dz_j\wedge d\bar{z}_j+ \sqrt{-1}|z_1|^{2\beta-2} dz_1 \wedge d\bar { z}_1.
\]
\end{enumerate}
\end{defi}

For any $\beta\in (0,1]$, let  $\hat{\cal H}_{\beta}$ be the space of all potentials $\varphi$ such that  $\omega_0+\sqrt{-1}\p\bar \p \varphi$ 
is a  K\"ahler metric on $X$ with cone angle $2\pi\beta$ along $D$.  It is well-known that,
for any $\varphi \in \cal H$, and for $\epsilon$ small enough (which may depend on $\varphi$),  we have
\[
\varphi + \epsilon |S|^{2\beta}_{h} \in \hat{\cal H}_{\beta}, 
\]
where $h$ is  a smooth Hermitian metric on $-K_X$ with Ricci curvature $\omega_0$ and $S$ is the defining section of $D$.

 There are other definitions of metrics with cone singularities which, for K\"ahler-Einstein metrics, turn out to be equivalent. One definition is to require that a local K\"ahler potential lies in a version of H\"older space ${\cal C}^{2,\gamma, \beta}\; $  for $0<\gamma < {1\over {\beta}} - 1$ \cite{Dona11}.  The equivalence with Definition 1.3, for K\"ahler-Einstein metrics, is proved in \cite{JMR}  (Theorem 2). Higher regularity is also established in \cite{JMR}, including the fact that such metrics have an asymptotic expansion about  points on the divisor.

\section{Proof of Theorem \ref{thm1}} \label{sec 2}

Suppose $\omega_{\varphi_\beta}$ is a K\"ahler-Einstein metric on $X$ with cone angle 
$ 2\pi \beta$ along a smooth anti-canonical divisor $D$. Let $S$ be the defining section of $D$ in $H^0(X, -K_X)$.  By \cite{Dona11},  \cite{JMR}, it satisfies 
the following Monge-Amp\`ere equation
\begin{equation}
\omega_{\varphi_\beta}^{n} = e^{-\beta \varphi_{\beta} + h_{\omega_0}} {\omega_0^{n}\over {|S|_{h}^{2(1-\beta)}}},\qquad \qquad {\rm on}\;\; X\setminus D. \label{SKE:1}
\end{equation}
where $h_{\omega_0}$ is the Ricci potential of $\omega_0$ and we have chosen the normalization of $\varphi_\beta$ so that
$$\int_X e^{-\beta \varphi_{\beta} + h_{\omega_0}} {\omega_0^{n}\over {|S|_{h}^{2(1-\beta)}}}=\int_X \omega_0^n.$$
  To prove Theorem \ref{thm1}, we need to achieve the following
three goals simultaneously:
\begin{enumerate}
\item Approximate $\omega_{\varphi_\beta}$ by smooth K\"ahler metrics on $X\;$ locally smoothly away from $D$;
\item  The Ricci curvature of this sequence of metrics is positive and diameter is uniformly bounded from above;
\item  The Gromov-Hausdorff limit of this sequence of metrics is precisely the metric $(X, \omega_{\varphi_\beta})$.
\end{enumerate}

To achieve the first goal, we want to smooth the volume form of $\omega_{\varphi_\beta}$ first, and then use the Calabi-Yau theorem on $X\;$ to smooth the potential $\varphi_\beta$. Fix $p_0\in (1, (1-\beta_0)^{-1})$. Note that the volume form of $\omega_{\varphi_\beta}$ is bounded in $L^{p_0}$.  We can find a family of smooth
volume forms $\eta_\epsilon (\epsilon \in (0, 1])$ with $\int_X\eta_\epsilon=\int_X \omega_0^n$, which converges to $\omega_{\varphi_\beta}^n $ strongly in $L^{p_0}$ 
and  smoothly away from $D$.  For each $\eta_\epsilon$, by the Calabi-Yau theorem we can find a
smooth K\"ahler  potential $\varphi_\epsilon \in \cal H$ such that 
\[
     \omega_{\varphi_\epsilon}^n = \eta_\epsilon.
\]

  Following \cite{Kolo98},  we obtain  a uniform bound on $||\varphi_\epsilon||_{C^\gamma(X)}$ for some $\gamma\in (0,1)$.   This bound  and $\gamma$ depend only on $X, D$, $\omega_0$ and the $L^{p_0}$ norm of $\omega_{\varphi_\beta}^n\over \omega_0^n$.  Furthermore $\{\varphi_\epsilon\}$ converges by sequence to $\varphi_\beta$ in $C^{\gamma'}(X)$ for some $\gamma'$ slightly smaller than $\gamma$.

To achieve our second goal, we need to modify the volume form to secure positive Ricci curvature. 
Following Yau \cite{Yau78}, we can solve the following equation for $\epsilon \in (0, 1]:\;$ 
\begin{equation}
\omega_{\psi_{\epsilon}}^{n } =  e^{-\beta \varphi_{\epsilon} + h_{\omega_0}} { \omega_0^n \over {(|S|_{h}^{2} + \epsilon)^{1-\beta}}}.   \label{SKE:3}
\end{equation}
Here  we need to normalize $\varphi_\epsilon$ so that
$$\int_{X}\; e^{-\beta \varphi_{\epsilon} + h_{\omega_0}} { \omega_0^n \over {(|S|_{h}^{2} + \epsilon)^{1-\beta}}}  = \int_X \omega_{0}^{n}.$$

In Theorem 8 of \cite{Yau78}, Yau treated a more general case with meromorphic right hand side.   Note our initial approximation $\varphi_\epsilon$ is smooth but does not have high regularity control outside $D$, and we will discuss a bit more after Proposition \ref{prop2.3}. There are some similarities in our approach and the work of  Campana, Guenancia, and Paun \cite{CGP11}.   For more recent work on complex Monge-Amp\`ere equation on K\"ahler manifolds and generalizations,  we refer to \cite{EGZ} for further references.  \\ 

A direction calculation shows that

\begin{prop} The Ricci form of $\omega_{\psi_{\epsilon}}$ approximates $\beta \omega_\beta + (1-\beta) [D] $  as $\epsilon \rightarrow 0$. Moreover,
\[
Ric(\omega_{\psi_{\epsilon}} ) \geq \beta \omega_{\varphi_{\epsilon}}> 0, \qquad \forall\  \epsilon \in (0, 1].
\]\label{prop2.1}
\end{prop}

\begin{proof}  For any smooth function $f > 0$, we have (c.f. \cite{Yau78})
\[
  \begin{array}  {lcl} \sqrt{-1}\p \bar \p \log (f + \epsilon)  & = & \sqrt{-1}\p {{\bar \p f}\over {f+\epsilon}}  \\
  & = &   {{\sqrt{-1}\p \bar \p f}\over {f+\epsilon}} -{{\sqrt{-1}\p f \wedge \bar \p f}\over {(f+\epsilon)^{2}}} 
  \\ & = & {f \over {f+\epsilon}} \left(  {{\sqrt{-1}\p \bar \p f}\over {f}} -{{\sqrt{-1}\p f \wedge \bar \p f}\over {f^{2}}}  + {{\sqrt{-1}\p f \wedge \bar \p f}\over {f^{2}}}\right)  -{{\sqrt{-1}\p f \wedge \bar \p f}\over {(f+\epsilon)^{2}}} \\
  &  = & {f \over {f+\epsilon}}\sqrt{-1} \p \bar \p \log f   + \epsilon {{\sqrt{-1}\p f \wedge \bar \p f} \over {f (f+\epsilon)^{2}}} 
  \\& \geq & {f \over {f+\epsilon}}\sqrt{-1} \p \bar \p \log f.  \end{array}
\]

Using this, we can calculate the Ricci form on $X\setminus D$: 

 \[
\begin{array}{lcl}  Ric (\omega_{\psi_{\epsilon}})
& = &   -  \sqrt{-1} \p \bar \p h_{\omega_0}  +  (1-\beta) \sqrt{-1} \p \bar \p \log (|S|^{2}_{h} +\epsilon) + Ric(\omega_0) + \beta \sqrt{-1} \partial \b \partial \varphi_{\epsilon}  \\
& = & \omega_0 + (1-\beta) \sqrt{-1} \p \bar \p \log (|S|^{2}_{h} +\epsilon) + \beta  \sqrt{-1} \partial \b \partial \varphi_{\epsilon}\\
&=& \beta\omega_{\varphi_\epsilon}+(1-\beta)(\omega_0+\sqrt{-1} \p \bar \p \log (|S|^{2}_{h} +\epsilon))\\
&=& \beta\omega_{\varphi_\epsilon}+(1-\beta)\frac{\epsilon}{|S|^2_h+\epsilon} \omega_0\\
&\geq& \beta\omega_{\varphi_\epsilon}.
\end{array}
\]
Since $\psi_\epsilon$ is smooth this also holds on the whole $X$. 
\end{proof}
For later purpose we denote	 
 \begin{equation}
 (1-\beta) \chi_{\epsilon} = Ric (\omega_{\psi_{\epsilon}}) -\beta \omega_{\varphi_{\epsilon}}. 
 \label{positiveform1}
 \end{equation}
By the previous calculation, this converges to  $2\pi (1-\beta)[D]$ in the sense of currents.  

Now we derive estimates on $\omega_{\psi_\epsilon}$. We make the convention that unless otherwise emphasized all constants appearing below are positive and depend only on $X, D, \omega_0,  \omega_{\varphi_\beta}$. Also the norms of the functions appearing in this article are always taken with respect to the background metric $\omega_0$. 

\begin{theo} There exists a uniform constant $C_1>0$ such that  for any $\epsilon \in (0, 1]$,
\[
 C_1^{-1}\omega_0 < \omega_{\psi_{\epsilon}} \leq {C_1\over (\epsilon + |S|^{2}_{h})^{1-\beta}} \cdot \omega_0.
\]\label{thm2.2}
\end{theo}
\begin{proof}

First,  by construction we have a constant $c_1$ and $p_0>1$ such that 
$$||\frac{\omega_{\psi_\epsilon}^n}{\omega_0^n}||_{L^{p_0}}\leq c_1. $$
By a theorem of  Kolodziej \cite{Kolo98}, this implies $||\psi_{\epsilon}||_{C^{\gamma}(X)}\leq c_2$ for some $\gamma \in(0,1)$.
To derive $C^{2}$ estimate, we view the identity map $id: (X, \omega_{\psi_{\epsilon}}) \rightarrow (X, \omega_0)$  as a harmonic map
with energy density
\[
e(\psi_\epsilon) = \tr_{\omega_{\psi_\epsilon}} \omega_0
\]
 Then, using the fact that $Ric(\omega_{\psi_\epsilon})> 0$ and $Rm(\omega_0)\leq c_3$,  $e(\psi_\epsilon)$ satisfies the following Chern-Lu differential inequality \cite{Lu67} (c.f. also \cite{JMR}): 
 \[
\triangle_{\psi_\epsilon} (\log e(\psi_\epsilon)-c_4\psi_\epsilon)\geq c_5 e(\psi_\epsilon)-c_6.
\]
Since $\psi_\epsilon$ is uniformly bounded by $c_2$,   
  by maximum principle, we have
 \[
 e(\psi_\epsilon) \leq c_6
 \]
 or
 \[
 c_6^{-1} \omega_0 \leq \omega_{\psi_\epsilon}.
 \]
 Plugging this into the  Monge-Amp\`ere equation (\ref{SKE:3}), we obtain
 
 \begin{equation}
   C_1^{-1} \omega_0 \leq \omega_{\psi_\epsilon} \leq  C_1\cdot ( |S|^{2}_{h} + \epsilon)^{- (1-\beta)} \omega_0.
   \label{eq:metriccontrol2}
 \end{equation}
 \end{proof}
 
 It follows that we have a uniform bound on $\Delta_{\omega_0}\psi_\epsilon$ locally away from $D$, and  a  global uniform bound on $||\psi_\epsilon||_{C^\gamma(X)}$.  So  $\{\psi_{\epsilon}\}$ by sequence converges to a limit potential $\psi_0$ globally  in $C^{\gamma'}(X)$ and in $C^{1, \alpha}$ locally away from $D$.

\begin{prop} We
 have $\psi_{0} = \varphi_{\beta} + \emph{constant}$. \label{prop2.3}
\end{prop}

\begin{proof}  This follows directly by the general uniqueness theorem for Monge-Amp\`ere equation \cite{Bedford76, Kolo03, Blocki}.  For the convenience of readers, we give a detailed account in our special case. Since 
\[
\omega_{\psi_{\epsilon}}^{n} - \omega_{\varphi_{\epsilon}}^{n}
\]
converges to $0$ as $\epsilon \rightarrow 0$  in $L^{p} $ topology for some fixed $p> 1$, we have
\[
\int_{X} \; (\varphi_{\epsilon} -\psi_{\epsilon}) (\omega_{\psi_{\epsilon}}^{n} -\omega_{\varphi_{\epsilon}}^{n}) \rightarrow 0.
\]
It follows that

\[
\begin{array}{lcl}
 & &\int_{X}\; \sqrt{-1} (\partial \psi_{\epsilon} - \partial \varphi_{\epsilon}) \wedge  (\bar \partial \psi_{\epsilon} - \bar \partial \varphi_{\epsilon}) \wedge
 \displaystyle \sum_{k=0}^{n-1} \omega_{\psi_{\epsilon}}^{k} \omega_{\varphi_{\epsilon}}^{n-1-k} \\
 & &   \qquad \qquad= -  \int_{X}\; (\psi_{\epsilon} -\varphi_{\epsilon}) (\omega_{\psi_{\epsilon}}^{n} - \omega_{\varphi_{\epsilon}}^{n}) \rightarrow 0.
\end{array}
\]
Positivity of the integrand means that for any $\delta>0$, we have
\[
\int_{X \setminus D_{\delta}}\; \sqrt{-1}(\partial \psi_{\epsilon} - \partial \varphi_{\epsilon}) \wedge  (\bar \partial \psi_{\epsilon} - \bar \partial \varphi_{\epsilon}) \wedge
 \displaystyle \sum_{k=0}^{n-1} \omega_{\psi_{\epsilon}}^{k} \omega_{\varphi_{\epsilon}}^{n-1-k} \rightarrow  0,
\]
where $D_{\delta}$ is the $\delta$-tubular neighborhood of $D$, defined by the metric $\omega_0$. Since every term is non-negative, we have
\[\int_{X \setminus D_{\delta}}\; \sqrt{-1}(\partial \psi_{\epsilon} - \partial \varphi_{\epsilon}) \wedge  (\bar \partial \psi_{\epsilon} - \bar \partial \varphi_{\epsilon}) \wedge \omega_{\psi_{\epsilon}}^{n-1}  \rightarrow  0,
\]
 By Theorem \ref{thm2.2},  $ \omega_{\psi_{\epsilon}} \geq C_1^{-1}\omega_0$. It follows
that 
\[\int_{X \setminus D_{\delta}}\; \sqrt{-1}(\partial \psi_{\epsilon} - \partial \varphi_{\epsilon}) \wedge  (\bar \partial \psi_{\epsilon} - \bar \partial \varphi_{\epsilon}) \wedge \omega_{0}^{n-1}  \rightarrow  0,
\]
In other words, 
\[
\partial \psi_{0} - \partial \varphi_{\beta} = 0, \qquad {\rm in}\;\; X\setminus D_{\delta}.
\]
So 
\[
\psi_{0} = \varphi_{\beta}+\text{constant}, \qquad {\rm in}\;\; X\setminus D_{\delta}.
\]
 Since $\delta$ is arbitrary,  this finishes the proof.

\end{proof}

To obtain more regularity control, we can substitute $\varphi_\epsilon$ in Equation  (\ref{SKE:3}) by $\psi_\epsilon.\;$ Then, the right
hand side of the equation will have a uniform $C^{1,1}$ bound (i.e. bound on its Laplacian $\Delta_{\omega_0}$) locally away from $D$.  Then we can repeat
the same procedure to obtain a new $\psi_\epsilon'$ for $\epsilon \in(0,1]$.  As before, this  new sequence  $\psi'_\epsilon$
 converges to $\varphi_\beta$ globally in $C^{\gamma'}$ and they satisfy
the same estimate as in Theorem \ref{thm2.2}. 
 Now following standard theory of Evans-Krylov \cite{Evans} \cite{Krf} and bootstrapping,   we can obtain an interior $C^{3,\gamma}$ estimate on $X\setminus D$ for some $\gamma\in( 0, 1).\;$ It follows that  $\{\psi_{\epsilon}'\}$ by sequence converges in $C^{3, \gamma'}$ to $\varphi_\beta$  locally away from $D$. For simplicity from now on we will denote by $\psi_\epsilon$ this new sequence. If we keep running this procedure, we get even higher
 derivative control away from $D$. \\

To achieve the second goal, we  need to prove the following proposition.

\begin{prop} For any $\epsilon \in (0, 1]$, the diameter of $(X, \omega_{\psi_{\epsilon}})$ is uniformly bounded above by  a constant $C_2$. 
\label{prop2.4}
\end{prop}
\begin{proof}  Since $D $ is smooth,  there exists a small constant $\delta > 0$ such
that the restriction of the background metric $\omega_0$ to the $\delta$-\emph{tubular} neighborhood of $D$, denoted by $D_{\delta} ,$ is  equivalent to the product metric on $D\times \mathbb B$, where $\mathbb B$ is the the standard disc of radius $\delta$. Following the estimate
in Theorem \ref{thm2.2}, for every point in $D_{\delta}$, there is a curve connecting it  to $\p D_{\delta}$ with length bounded by $ c_1 \delta^{\beta}.\;$
On the other hand, the varying metric is bounded above by the metric $c_2\delta^{\beta-1}\omega_0$ in $X\setminus D_{\delta}$.
Therefore, the diameter of $(X, \omega_{\psi_{\epsilon}})$ is controlled above by $
c_3(\delta^{\beta} + \delta^{\beta - 1}).$ 

\end{proof}

Finally, to achieve the third goal,  we need to study the Gromov-Hausdorff limit of a sequence of Riemannian manifolds with
positive Ricci curvature.

\begin{prop}   $(X, \omega_{\varphi_{\beta}})$ is the Gromov-Hausdorff limit
of $(X, \omega_{\psi_{\epsilon}})$ as $\epsilon\rightarrow 0$.\label{prop2.5}
 \end{prop}
\begin{proof} By Theorem \ref{thm2.2} and Proposition \ref{prop2.4},  we have the following:
\begin{enumerate}

\item $ \omega_{\psi_{\epsilon}}$ converges as a current to $\omega_{\varphi_{\beta}};$
\item $(X\setminus D, \omega_{\psi_{\epsilon}})$ converges locally in $C^{3, \gamma'}$ to $(X\setminus D, \omega_{\varphi_{\beta}});\;$
\item Any fixed $\delta$-tubular neighborhood $D_\delta$ of $D$ with respect
to the background metric $\omega_0$ is contained in a $\eta(\delta)$-tubular neighborhood of $D$ with respect to the varying
metric $\omega_{\psi_{\epsilon}}$, where $\eta(\delta)$ tends to zero as $\delta\;$ tends to zero. \end{enumerate}

To prove the desired Gromov-Hausdorff convergence, we use the identity map from $X$ to itself, for any $\delta>0$ small, we want to show
that if $\epsilon$ is small enough then
\begin{equation}
 |d_{\psi_\epsilon}(x, y) - d_{\varphi_\beta}(x,y) | < \delta, \qquad \forall x, y \in X.
 \label{eq:GHdistance}
\end{equation}
Fix a small constant 
\[
\delta_1\ll \delta^{1\over \beta}.
\]
For any two points outside $D_{\delta_1}$, for $\epsilon>0$ sufficiently small, the preceding inequality (\ref{eq:GHdistance}) obviously holds with ${\delta \over 3}$ in the right hand side. Now for any two points $x,y \in D_{\delta_1},\;$
there exist two points $x_1, y_1 \in \p D_{\delta_1}$ such that
\[
d_{\varphi_\beta} (x_1, x)  + d_{\varphi_\beta} (y_1, y)  < {\delta \over 3} \qquad {\rm and}\; \;\;d_{\psi_\epsilon} (x_1, x)  + d_{\psi_\epsilon} (y_1, y)  < {\delta \over 3}.\]
A simple triangle inequality implies that
\[
|d_{\psi_\epsilon}(x, y) - d_{\varphi_\beta}(x,y) | <  {\delta \over 3} +  {\delta \over 3} +  {\delta \over 3} = \delta.
\]
\end{proof}

To finish the proof of Theorem \ref{thm1}, what is left to prove is the uniform diameter bound (depending only on $\beta_0$). To see this, one notices that for a K\"ahler metric with cone singularities along $D$, the complement $X\setminus D$ is geodesically convex, so we can apply Myers' theorem to see that the diameter of $\omega_\beta$ is uniformly bounded by $\pi \sqrt{\frac{n-1}{\beta}}\leq \pi \sqrt{\frac{n-1}{\beta_0}}$. Then for each $\beta \geq \beta_0$, we  can apply  Proposition \ref{prop2.5} (choosing $\epsilon$ sufficiently small)  to obtain a sequence of smooth K\"ahler metrics of positive Ricci curvature with  Gromov-Hausdorff limit  $(X, \omega_{\varphi_\beta})$  and the diameter of this sequence of metrics is uniformly bounded above by $2\pi \sqrt{\frac{n-1}{\beta_0}}$.

\section{Proof of Theorem \ref{thm2}}
To prove Theorem \ref{thm2}, we need to achieve the fourth goal: to approximate the K\"ahler-Einstein metric $\omega_{\varphi_\beta}$ by smooth K\"ahler metrics
with Ricci curvature bounded from below by some uniform positive number.  From the complex Monge-Amp\`ere theory, this is very different from the first and second goal  where we can obtain $C^{0}$ estimate via Kolodziej's theorem \cite{Kolo98} directly. What
we shall do is to use the metric constructed in the previous section as a starting point, and use continuity
method to solve the twisted K\"ahler-Einstein equation up to $t = 1-\beta.\;$ To do this, we need to obtain a uniform $C^0$ estimate. The key observation is that for the family of $(1,1)$ forms $\chi_\epsilon$ (see Equation (\ref{positiveform1})) which converges to $2\pi (1-\beta) [D],\;$ the twisted K-energy
$E_{\epsilon,(1-\beta) D}$ dominates the K-energy  $E_{(1-\beta) D}$ from above (See
formulas (\ref{twisted energy 1}), (\ref{twisted energy 2}) and Lemma \ref{lem3.5} below for precise statements).\\

Following Szekelyhidi \cite{G}, we define
 \[
R(X) = \displaystyle \sup \{t: \exists \  \omega'  \in 2\pi c_1(X)\ {\rm such\;\; that }\; Ric(\omega') > t \omega'\}.
\]

Theorem \ref{thm2} is a consequence of the following

\begin{theo}  If there is a K\"ahler-Einstein metric $\omega_{\varphi_{\beta}}$ with cone angle $ 2\pi \beta\;$ along $D$,  then
 $(X, \omega_{\varphi_{\beta}})$ is the Gromov-Hausdorff limit of a sequence of K\"ahler metrics with
Ricci curvature bounded below by $\beta>0.\;$ In particular, $R(X)\geq\beta$.
\label{thm3.1} 
\end{theo}

This verifies one aspect of a conjecture by the second named author earlier \cite{Dona11}.  
We need to do some preparation first.  
Set
\[
\chi = \sqrt{-1}\partial \b \partial \log |S|^{2}_{h} + \omega_0.
\]
By the Poincar\'e-Lelong equation, this is the same as the current $2\pi [D]$.  For  $\epsilon>0$ sufficiently small,  we define (c.f. Equation (\ref{positiveform1})):
\[
\chi_{\epsilon} =\sqrt{-1} \partial \b \partial \log (|S|^{2}_{h} +\epsilon)  + \omega_0 > 0
\]
Clearly as currents
\[
\chi_{\epsilon} \rightarrow \chi.
\] 

For any $\varphi\in \cal H$, we choose a smooth family of potentials $\varphi(t)  (t\in [0, 1])$ in $\cal H$ with $\varphi(0)=0$ and $\varphi(1)=\varphi$. We denote $\omega_t=\omega_0+\sqrt{-1}\partial \bar{\partial}\varphi(t)$.  For any smooth function $f(t, \cdot)$, we write $\dot {f}(t) $ for the time derivative ${{\p f} \over {\p t}}(t, \cdot)$.
\begin{defi} Define a functional $J_{\chi_{\epsilon}}$ by  
\[
 J_{\chi_{\epsilon}}(\varphi)=  n\int_0^1 dt \int_{X} \; \dot{\varphi}(t) ( \chi_{\epsilon} - \omega_{t}) \wedge \omega_{t}^{n-1}, 
\]\label{def3.2}
and as $\epsilon$ tends to zero,  we define a functional $J_{\chi}$ by 
\[
 J_{\chi}(\varphi) =2\pi n \int_0^1dt \int_D \dot{\varphi}(t) \omega_{t}^{n-1}- n\int_0^1dt \int_{X} \;\dot{\varphi}(t)  \omega_{t}^{n}, 
\]
\end{defi}
\begin{defi}  Define the K-energy functional $E$ by
\[
E(\varphi)  = - n\int_0^1dt \int_{X}\; \dot{\varphi}(t)  (Ric(\omega_t) - \omega_{t}) \wedge \omega_{t}^{n-1}.
\]\label{def3.3}
\end{defi}
 One can check these  do not depend on the choice of the path and hence are well-defined functionals on $\cal H$. Of course  they are only  defined up to an additive constant. We have fixed this constant by imposing that they have value $0$ at $\varphi=0$.
Then,  we  define the twisted K-energy 
\begin{equation}
E_{\epsilon,(1-\beta) D}(\varphi) = 
E(\varphi) + (1-\beta) J_{\chi_{\epsilon}}(\varphi), \label{twisted energy 1}
\end{equation}
and
\begin{equation}
E_{(1-\beta) D}(\varphi) = 
E(\varphi) + (1-\beta) J_{\chi}(\varphi).  \label{twisted energy 2}
\end{equation}
First, we  give an explicit formula for $J_{\chi_{\epsilon}}\;$ (c.f.  \cite{LiSun} ).
\begin{prop} We have
\begin{equation} J_{\chi_{\epsilon}}(\varphi)  = \int_{X}\;  \log (|S|^{2}_{h} + \epsilon)\cdot (\omega_{\varphi}^{n} -\omega_0^{n}) + n  \int_{0}^{1} d\, t \int_{X}\; {\dot{\varphi}(t)}( \omega_0-\omega_{t})\wedge \omega_{t}^{n-1}.   \label{twisted energy 3} \end{equation}
\label{prop3.4}
\end{prop}
\begin{proof} 
This is  a direct calculation:
\[
\begin{array}{lcl} 
&& J_{\chi_{\epsilon}}(\varphi)\\ 
& = & n\int_{0}^{1} dt \int_{X}{\dot{\varphi}(t)} \chi_{\epsilon}\wedge \omega_{t}^{n-1}-n \int_{0}^{1} d\, t \int_{X}\; {\dot{\varphi}(t)} \ \omega_{t}^{n}\\
& = & n \int_{0}^{1} d\, t \int_{X}\; {\dot{\varphi}(t)}\sqrt{-1}\partial \bar \partial \log (|S|^{2}_{h} + \epsilon)\wedge \omega_{t}^{n-1} +   n\int_{0}^{1} d\, t \int_{X}\; {\dot{\varphi}(t)} (\omega_0-\omega_t)\wedge \omega_{t}^{n-1} 
\\
& = & n \int_{0}^{1} d\, t \int_{X}\;  \log (|S|^{2}_{h} + \epsilon)\cdot \triangle_{\omega_t} {\dot{\varphi}(t)} \omega_{t}^{n} +   n\int_{0}^{1} d\, t \int_{X}\; {\dot{\varphi}(t)}(\omega_0-\omega_t)\wedge \omega_{t}^{n-1} \\
& = & \int_{0}^{1} d\, t \int_{X}\;  \log (|S|^{2}_{h} + \epsilon)\cdot   {\partial \over {\partial t}}  \omega_{t}^{n} +   n\int_{0}^{1} d\, t \int_{X}\; {\dot{\varphi}(t)}(\omega_0-\omega_t)\wedge \omega_{t}^{n-1} \\
& = & \int_{X}\;  \log (|S|^{2}_{h} + \epsilon)\cdot (\omega_{\varphi}^{n} -\omega_0^{n}) + n  \int_{0}^{1} d\, t \int_{X}\; {\dot{\varphi}(t)} ( \omega_0-\omega_t)\wedge \omega_{t}^{n-1}.
\end{array}\]
\end{proof}
Similarly, we have
\[
J_{\chi}(\varphi) =  \int_{X}\;  \log |S|^{2}_{h} \cdot (\omega_{\varphi}^{n} -\omega_0^{n}) + n  \int_{0}^{1} d\, t \int_{X}\; {\dot{\varphi}(t)}  (\omega_0-\omega_t)\wedge \omega_{t}^{n-1}.
\]
Now we have the following observation
\begin{lem} There exists a uniform constant $C_3=C_3(X, D, \omega_0)$ such that for any $\epsilon \in (0,1]$ and for any smooth K\"ahler potential $\varphi$,
we have
\[
J_{\chi_{\epsilon}} (\varphi) \geq J_{\chi} (\varphi) - C_3.
\]
As a consequence, we also have
\[
E_{\epsilon, (1-\beta) D}(\varphi) \geq E_{(1-\beta)D}(\varphi) - C_3.
\]
\label{lem3.5}
\end{lem}
\begin{proof}  This follows from an elementary calculation:

\[
\begin{array}{lcl} J_{\chi_{\epsilon}}(\varphi)  & = & \int_{X}  \log (|S|^{2}_{h} + \epsilon)\cdot (\omega_{\varphi}^{n} -\omega_0^{n}) +   n\int_{0}^{1} d\, t \int_{X}\; {\dot{\varphi}(t)} (\omega_0-\omega_t) \wedge \omega_{t}^{n-1} \\
& =  &  \int_{X}\;  \log (|S|^{2}_{h} + \epsilon)\cdot \omega_{\varphi}^{n} -  \int_{X}\;  \log (|S|^{2}_{h} + \epsilon)\cdot\omega_0^{n} +  n\int_{0}^{1} d\, t \int_{X}\; {\dot{\varphi}(t)} (\omega_0-\omega_t)\wedge \omega_{t}^{n-1} \\
& \geq & \int_{X}\;  \log |S|^{2}_{h} \cdot \omega_{\varphi}^{n}  -  \int_{X}\;  \log |S|^{2}_{h}\cdot\omega_0^{n}   -  \int_{X}  {\log \frac{|S|^2_h+\epsilon}{|S|^2_h}}\cdot\omega_0^{n} +n\int_{0}^{1} d\, t \int_{X}\; {\dot{\varphi}(t)} (\omega_0-\omega_t)\wedge \omega_{t}^{n-1} \\
& \geq & \int_{X}\;  \log |S|^{2}_{h} \cdot (\omega_{\varphi}^{n} -\omega_0^{n}) +  n\int_{0}^{1} d\, t \int_{X}\;{\dot{\varphi}(t)} (\omega_0-\omega_t)\wedge \omega_{t}^{n-1}  -C_3\\
 &= & J_{\chi} (\varphi) - C_3.
\end{array}
\]
\end{proof}
It is well-known that the K-energy has an explicit expression (c.f. \cite{Chen}): 
\[
E(\varphi)=\int_X \log \frac{\omega_\varphi^n}{\omega_0^n}\omega_0^n+I(\varphi)+Q(\varphi),
\]
where 
\[
I(\varphi)= n\int_0^1dt \int_X {\dot{\varphi}(t)} \omega_t^n,
\]
and 
\[
Q(\varphi)=-n\int_0^1dt \int_X{\dot{\varphi}(t)} Ric(\omega_0)\wedge \omega_t^{n-1}.
\]

Following  \cite{G} in the smooth case and  \cite{LiSun} in  the case with cone singularities, we have

\begin{prop} If there exists a K\"ahler-Einstein metric $\omega_{\varphi_{\beta}}$ with cone angle $2\pi \beta > 0$
along $D, $ then the twisted K-energy $E_{(1-\beta)D}$ is proper on $\cal H$. In other words, there are constants $C_4, C_5$ depending on $X, D, \omega_0$ and $\beta$ such that for any smooth K\"ahler potential $\varphi$, we have
\[
E_{(1-\beta)D}(\varphi) \geq C_{4}\cdot J_{0}(\varphi) - C_{5}, 
\]
where 
\[
J_{0}(\varphi)=\int_X \varphi (\omega_0^n-\omega_\varphi^n). 
\]
\label{prop3.6}
\end{prop}
\begin{proof} By \cite{SW12} there is no non-trivial holomorphic vector field on $X$ that is tangential to $D$. So by the openness theorem \cite{Dona11} for a slightly larger $\beta' > \beta$ there exists a K\"ahler-Einstein metric on $X$ with cone angle $2\pi\beta'$ along $D$. By  \cite{Berndt11} and \cite{B}, the twisted Ding functional is bounded from below on $\cal H$. By \cite{Li08}, the infimum of the Ding functional and the infimum of the K-energy are the same in the anti-canonical class. This is generalized in \cite{B} to the case with cone singularities. It follows that the
twisted K-energy $E_{(1-\beta')D}$ is bounded from below on $\cal H$. On the other hand, it is proved in \cite{B} that for $\beta''>0$ sufficiently small $E_{(1-\beta'')D}$ is proper on $\cal H$.  Since the twisted K-energy is linear in $\beta$ \cite{LiSun}, we see that $E_{(1-\beta)D}$ is proper on $\cal H$. 
\end{proof}

To prove Theorem \ref{thm3.1}, we need to set up a continuity path. 
For any $\epsilon>0$, we consider the following equation for $ \phi_\epsilon(t)\in {\cal H} (t\in [0, \beta])$:
\begin{equation} \label{eqn3.4}
Ric (\omega_{\phi_{\epsilon}(t)}) = t \omega_{\phi_{\epsilon}(t)} + (\beta-t) \omega_{\varphi_{\epsilon}} + (1-\beta) \chi_{\epsilon}. 
\end{equation}
This is equivalent to the complex Monge-Amp\`ere equation:
\begin{equation}\label{SKE:4}
\left\{
                                                               \begin{array}{ll}
\omega_{\phi_{\epsilon}(t, \cdot)}^{n} = e^{- t\phi_{\epsilon}(t, \cdot) -(\beta - t) \varphi_{\epsilon} + h_{\omega_0}} {1 \over {(|S|^{2}_{h} +\epsilon)^{1-\beta}}} \omega_0^{n},   & \\
          \phi_{\epsilon}(0, \cdot) =   \psi_{\epsilon}.  &\\

                                                               \end{array}
                                                             \right.
\end{equation}
Here $\psi_{\epsilon}$ is the solution in Equation (\ref{SKE:3}); in other words,  we have
\[
\omega_{\psi_{\epsilon}}^{n} = e^{ -\beta \varphi_{\epsilon} + h_{\omega_0}} {1 \over {(|S|^{2}_{h} +\epsilon)^{1-\beta}}} \omega_0^{n}.
\]
If we can solve Equation (\ref{SKE:4}) up to $t =\beta, $  then we have:
\[
Ric(\omega_{\phi_{\epsilon}(\beta)} )= \beta \omega_{\phi_{\epsilon}(\beta)} + (1-\beta) \chi_{\epsilon} \geq \beta \omega_{\phi_{\epsilon} (\beta)}.
\]
\begin{lem}  \label{lem3.7}
There is a constant $C_6>0$ such that
\[
|E(\epsilon, (1-\beta) D) (\psi_\epsilon)| \leq C_6.
\]
\end{lem}
\begin{proof} By Theorem \ref{thm2.2} this is  a direct verification using the explicit expressions of energy functionals described above. 
\end{proof}
\begin{lem} Along the continuity path, $E_{\epsilon, (1-\beta) D}(\phi_\epsilon(t))$ decreases monotonically. \label{lem3.9}
\end{lem}

\begin{proof} 
Here we follow   \cite{G}.  We take time derivative on both sides of Equation (\ref{SKE:4}). Then,  
\[
\triangle_{\phi_\epsilon} \dot{\phi_\epsilon}  = - t \dot{\phi_\epsilon} -(\phi_\epsilon -\varphi_{\epsilon}).
\]
In this calculation we have omitted the parameter $t$ for simplicity.
A straightforward calculation then follows,
\[
\begin{array} {lcl} 
&&{d\over {d\, t}} (E + (1-\beta) J_{\chi_{\epsilon}})(\phi_\epsilon(t))\\ 
&=& n \int_X \dot{\phi_\epsilon}(-Ric(\omega_{\phi_\epsilon})+\omega_{\phi_\epsilon}+(1-\beta)\chi_\epsilon-(1-\beta)\omega_{\phi_\epsilon})\wedge \omega_{\phi_\epsilon}^{n-1}\\
&=& n\int_X \dot{\phi_\epsilon}(-t\omega_{\phi_\epsilon}-(\beta-t)\omega_{\varphi_\epsilon}-(1-\beta)\chi_\epsilon+\beta \omega_{\phi_\epsilon}+(1-\beta)\chi_\epsilon)\wedge \omega_{
\phi_\epsilon}^{n-1}\\
& = & n(\beta -t ) \int_{X}\; \dot{\phi_\epsilon} (\omega_{\phi_\epsilon} - \omega_{\varphi_{\epsilon}}) \wedge \omega_{\phi_\epsilon}^{n-1} \\ 
& = &  n(\beta -t ) \int_{X}\;  (\phi_\epsilon -\varphi_{\epsilon}) \cdot \triangle_{\phi_\epsilon} \dot{\phi_\epsilon} \; \omega_{\phi_\epsilon}^{n}\\
& = &  -n (\beta -t ) \int_{X}\;(\phi_\epsilon -\varphi_{\epsilon}) \cdot ( t \dot{\phi_\epsilon} +(\phi_\epsilon -\varphi_{\epsilon}))\omega_{\phi_\epsilon}^n \\
& =  & - n(\beta -t) \int_{X} (\phi_\epsilon -\varphi_{\epsilon})^{2} \omega_{\phi_\epsilon}^{n} - nt (\beta -t) \int_{X} (\phi_\epsilon -\varphi_{\epsilon}) \cdot \dot{\phi_\epsilon}  \;\omega_{\phi_\epsilon}^{n} \\
& \leq  & n t (\beta -t) \int_{X}  (\triangle_{\phi_\epsilon} \dot{\phi_\epsilon} + t \dot{\phi_\epsilon}) \cdot \dot{\phi_\epsilon} \;\omega_{\phi_\epsilon}^{n}
\\& \leq & 0. \end{array}
\]
The last inequality holds because $Ric(\omega_{\phi_\epsilon(t, \cdot)}) > t \omega_{\phi_\epsilon(t, \cdot)} $ and  $\triangle_{\phi_\epsilon}  + t $ is a negative operator.
\end{proof}

Now we are ready to prove Theorem \ref{thm3.1}.
\begin{proof}  According to Proposition \ref{prop3.6}, the  twisted K-energy $E_{(1-\beta) D}(\varphi) $ is proper on $\cal H$. Following Lemma \ref{lem3.5}, 
\[
E_{\epsilon, (1-\beta) D}(\varphi) \geq E_{(1-\beta) D}(\varphi) - C_3
\]
is also proper on $\cal H.\;$ By monotonicity, we have
\[
E_{\epsilon, (1-\beta) D}(\phi_{\epsilon} (t)) \leq E_{\epsilon, (1-\beta) D}(\phi_{\epsilon}(0)) = E_{\epsilon, (1-\beta) D}(\psi_{\epsilon}) , \qquad \forall t\in [0,\beta].
\]
So by Lemma \ref{lem3.7}, 
\[
E_{\epsilon, (1-\beta) D}(\phi_{\epsilon} (t)) \leq C_6, \qquad \forall \;t \in [0,\beta].
\]
By definition of properness there is a constant $C_7$ with
\[
J_{0}(\phi_\epsilon(t)) \leq C_7.
\]
It follows from the standard argument that we can solve Equation (\ref{SKE:4}) up to $t=\beta$, and there is a constant $C_8$ such that 
\[
    \displaystyle \sup_{\epsilon \in (0,1]} \displaystyle \max_{t\in [0,\beta]}\;||\phi_{\epsilon}(t)||_{L^\infty} \leq C_8.
\]
As in the proof of Theorem \ref{thm2.2},  there is a constant $C_9$ such that 
\[
 C_9\omega_0< \omega_{\phi_{\epsilon}(t)} \leq {C_9\over (\epsilon + |S|^{2}_{h})^{1-\beta}} \cdot \omega_0, \forall\; t\in [0,\beta]\; \epsilon \in (0,1).
\]
As before following \cite{Kolo98} and Evans-Krylov theory to bootstrap regularity away from divisor, one can prove that  $\phi_{\epsilon}(t, \cdot)$ converges 
to $\phi_{0}(t,\cdot)\;$ globally in $C^{\gamma'}(X)$ and  locally in $C^{3, \gamma'}$ away from $D$. Moreover, in $X\setminus D$, it satisfies the equation
\[
\omega_{\phi_{0}(t,\cdot) }^{n} = e^{-t \phi_{0}(t,\cdot)- (\beta - t)\varphi_{\beta} + h_{\omega_0}} {1 \over |S|^{2-2\beta}_{h} }\omega_0^{n},\qquad \forall t\in [0,\beta].
\]
with
\[
\phi_{0}(0,\cdot) = \varphi_{\beta}.\;
\]
This can be written in a more concise form  as
\begin{equation}
\omega_{\phi_{0}}^{n} = e^{-t(\phi_{0} -\varphi_{\beta})}  \omega_{\varphi_{\beta}}^{n},\qquad t \in [0,\beta].
\label{eq:SKE4}
\end{equation}

Since $\phi_{\epsilon}(t, \cdot)$ is uniformly bounded, we see  $\phi_{0}(\beta, \cdot)$ is in the weak sense (c.f. \cite{B}) a K\"ahler-Einstein metric on $X$ with cone angle $2\pi\beta$ along $D$. By \cite{SW12}, there is no non-trivial holomorphic vector field on $X$ which is tangential to $ D.\;$ So we can use the  uniqueness theorem of Berndtsson \cite{Berndt11} to obtain
\[
   \phi_{0}(\beta,\cdot) = \varphi_\beta(\cdot), \forall t \in [0,\beta].
\]
Then Proposition  \ref{prop2.5} implies that  $(X, \omega_{\phi_{\epsilon}(\beta)})$ converges  in the Gromov-Hausdorff topology to
$(X, \omega_{\varphi_\beta})\; $ as $\epsilon\rightarrow0$. \\
\end{proof}

Here we give an alternative proof which makes use of the openness theorem proved by the second
named author, bypassing Berndtsson's theorem. 

\begin{proof}
  Note that the expression $e^{-t(\phi_0-\varphi_\beta)}$ in Equation (\ref{eq:SKE4}) is H\"older continuous on $X$, so it lies  in the H\"older space $\cal C^{,\gamma, \beta} $ for some $\gamma < {1\over \beta} -1$. 
  Following  \cite{Dona11}, the Laplacian operator $\triangle_{\varphi_{\beta}}$
defines a continuous and invertible  map 
$$\triangle_{\varphi_{\beta}}:\;\; {\cal C}_0^{2,\gamma, \beta}(X, D) \rightarrow {\cal C}^{,\gamma, \beta}(X, D).\;$$
Here ${\cal C}_0^{2,\gamma, \beta}(X, D)$ consists of functions in  ${\cal C}^{2,\gamma, \beta}(X, D)$ with zero average.
It follows that, 
there exists a continuous family $\psi(t,\cdot) \in {\cal C}_0^{2,\gamma, \beta}(X, D)$ for small $t$, say $ t \in [0,\epsilon_0]$, which solves
\[
\omega_{\psi}^{n} = e^{-t(\phi_{0} -\varphi_{\beta})}  \omega_{\varphi_{\beta}}^{n}
\]
and $\psi(0,\cdot) =\varphi_{\beta}.\;$
Then, either following the uniqueness in \cite{Kolo03}  or Proposition \ref{prop2.3}, we
have
\[
   \psi(t,\cdot) = \phi_{0}(t,\cdot)+\text{constant}
\]
for $t \in [0, \epsilon_0]$. It follows that $\phi_{0}(t,\cdot) \in {\cal C}^{2,\gamma, \beta} (X, D)$ 
for $t \in [0,\epsilon_0].\;$ \\

Now we define the constant family $\varphi_{\beta} (t, \cdot) = \varphi_{\beta, \cdot}, \;$ then it satisfies the same Equation (\ref{eq:SKE4}) for $t \in [0, \epsilon_{0}].\;$
\begin{equation}
\omega_{\varphi_{\beta}(t)}^{n} = e^{-t(\varphi_{\beta}(t) -\varphi_{\beta})}  \omega_{\varphi_{\beta}}^{n},\qquad t \in [0,\beta].
\label{eq:SKE5}
\end{equation}
By \cite{SW12}, there is no non-trivial holomorphic vector fields on $X$ which is tangential to $D. \;$ It follows that,
the first eigenvalue of $\omega_{\varphi_{\beta} } = \omega_{\phi_{0}(0,\cdot)}$ is strictly bigger than $\beta > 0.\;$ Consequently
for $t$ sufficiently small, $ \omega_{\phi_{0}(t,\cdot)}$ has eigenvalue strictly bigger than ${\beta \over 2}.\;$
Now compare the two families of  solutions to Equation (\ref{eq:SKE4}), by implicit function theorem again we see the uniqueness holds.
In other words, by making $\epsilon_0$ even smaller we have
\[
\phi_{0}(t,\cdot) = \varphi_{\beta}, \forall\ t \in [0, \epsilon_{0}].
\]
Repeating the same procedure as we increase  $t \leq \beta,$  we see the same holds for all $t\in [0,\beta].\;$
Our theorem is then proved.
\end{proof}

\begin{rem} \label{remark neg}
Finally we remark that in the case when $\lambda>1$  and $1-(1-\beta)\lambda\leq 0$ there is a complete existence theory \cite{JMR}, but the argument in this article also applies to prove the following 
\end{rem}
\begin{theo} 
Let $\lambda>1$ and $\beta_0\in (0, 1-\lambda^{-1}]$. If $\omega$ is a K\"ahler-Einstein metric with cone angle $2\pi\beta$  along $D\in |-\lambda K_X|$ with $\beta\in[\beta_0, 1-\lambda^{-1}]$,  then $(X,\omega)$ is the Gromov-Hausdorff limit of a sequence of smooth K\"ahler metrics $\omega_{i}$ with  
$ Ric(\omega_{i})\geq c_\beta \omega_{i}$ where
$c_\beta= (1-\lambda(1-\beta))\leq 0$, and diameter bounded above by a uniform constant depending only on $X, D$ and $\beta_0$. 
\end{theo}
\begin{proof} 
The main issue is that in our previous argument the diameter bound depends on the particular $\beta$, and in the case of positive Ricci curvature (as assumed before that $\beta_{0}> 1-\lambda^{-1}$) we can apply Myers' theorem to show that the bound only depends on $\beta_0$. Under our assumptions, we are in the  case of nonpositive Ricci curvature. In general the diameter can not have a uniform upper bound, if one varies the complex structure on $(X, D)$, even when $X$ has complex dimension one.  However for a fixed $(X, D)$, a closer look at the argument in the proof of Theorem \ref{thm2.2} and Proposition \ref{prop2.4} shows that the diameter bound really depends only on the lower bound on the Ricci curvature of $\omega_{\varphi_\beta}$ and the $L^\infty$ bound on $\varphi_\beta$ (which in turn depends on the  $L^{p_0}$ bound on the volume form of $\omega_{\varphi_\beta}$). Notice we assume $c_\beta\leq 0$, then the metric $\omega_{\varphi_\beta}$ satisfies the equation
\begin{eqnarray}
\omega_{\varphi_{\beta}}^{n} & = &  e^{-c_\beta\varphi_{\beta}+ h_{\omega_0}} {1 \over {|S|_h^{2(1-\beta)}}} \omega_0^{n}, \label{SKE:6}
\end{eqnarray}
Similar to the arguments in Section \ref{sec 2}, by Yau's theorem \cite{Yau78} for $\epsilon \in (0, 1]$ one can solve the equation for $\psi_\epsilon$:
\[
\omega_{\psi_\epsilon}^{n}  =   e^{-c_\beta\psi_\epsilon+ h_{\omega_0}} {1 \over {(|S|_h^2+\epsilon)^{(1-\beta)}}} \omega_0^{n}. \label{SKE:6}
\]
Direct calculation as before shows that $Ric(\omega_{\psi_\epsilon})\geq c_\beta \omega_{\psi_\epsilon}$. Moreover by the maximum principle 
 we see that there are constants $p_0\in (1, \frac{1}{1-\beta_0})$, and $A>0$ depending only on $X, D, \omega_0$ and $\beta_0$ such that for any $\epsilon\in (0, 1]$, 
 $$\sup_X \psi_\epsilon+||\frac{\omega_{\psi_\epsilon}^n}{\omega_0^n}||_{L^{p_0}}\leq A. $$
Following the arguments in Section \ref{sec 2}, one can show that as $\epsilon\rightarrow0$ the Gromov-Hausdorff limit of $(X, \omega_{\psi_\epsilon})$ is exactly $(X, \omega_{\varphi_\beta})$. Moreover, there is a uniform diameter bound independent of $\beta \in [\beta_0, 1-\lambda^{-1}]$. 
 \end{proof}


\begin{thebibliography}{20}

\bibitem{Bedford76} E. D. Bedford, T. A. Taylor. \emph{Uniqueness for the complex Monge-Amp\`ere equation for functions of logarithmic growth}.  Indiana Univ. Math. J. 38 (1989), no. 2, 455-469.
\bibitem{B} R. Berman. \emph{A thermodynamical formalism for Monge-Amp\`ere equations, Moser-Trudinger inequalities and K\"ahler-Einstein metrics.} arXiv:1011.3976.
\bibitem{Berndt11} B. Berndtsson. \emph{A Brunn-Minkowski type inequality for Fano manifolds and the Bando-Mabuchi uniqueness theorem.} arXiv:1103.0923.
\bibitem{Blocki} Z. Blocki. \emph{Uniqueness and stability for the Monge-Amp\`ere equation on compact K\"ahler manifolds}.  Indiana Univ. Math. J. 52 (2003), no. 6 1697-1701.
\bibitem{Calabi58} E. Calabi, \emph{Improper affine hyperspheres of convex type and a generalization of a theorem by K. J\"orgens.} Michigan Math. J. 5 1958 105-126.
\bibitem{CGP11} F. Campana, H.Guenancia, M. Paun, \emph{Metrics with cone singularities along normal crossing divisors and holomorphic tensor fields.} 	arXiv:1104.4879.
\bibitem{Chen} X-X. Chen. \emph{On the lower bound of the Mabuchi energy and its application.} Internat. Math. Res. Notices 2000, no. 12, 607-623.
\bibitem{CDS} X-X. Chen, S. Donaldson,  S. Sun. \emph{K\"ahler-Einstein metrics and stability}. arXiv:1210.7494.
\bibitem{Dona11} S. Donaldson. \emph{ K\"ahler metrics with cone singularities along a divisor.} arXiv:1102.1196.
\bibitem{Evans}  L. C. Evans, \emph{Classical solutions of fully nonlinear, convex, second-order elliptic equations}. Comm. Pure Appl. Math. 35 (1982), 333-363.
\bibitem{EGZ} P. Eyssidieux, V. Guedj, A. Zeriahi. \emph{ Singular K\"ahler-Einstein metrics}. J. Amer. Math. Soc. 22 (2009), 607-639.
\bibitem{GT} D. Gilbarg, N. Trudinger. \emph{Elliptic partial differential equations of second order.} Springer,1998.
\bibitem{JMR} T. D. Jeffres, R. Mazzeo, Y. Rubinstein. \emph{ K\"ahler-Einstein metrics with edge singularities.} 	arXiv:1105.5216.
\bibitem{Kolo98} S. Kolodziej. \emph{The complex Monge-Amp\`ere equation}. Acta Math. 180 (1998), no. 1, 69-117. 
\bibitem{Kolo03}  S. Kolodziej, \emph{The Monge-Amp\`ere equation on compact K\"ahler manifolds}. Indiana Univ. Math. J. 52 (2003), no. 3, 667-686.
\bibitem{Krf} N.V. Krylov, \emph{Boundedly nonhomogeneous elliptic and parabolic equations} (Russian).  Izv. Akad. Nauk SSSR Ser. Mat. 46 (1982), 487-523.
\bibitem{LiSun} C. Li, S. Sun. \emph{Conical K\"ahler-Einstein metric revisited. }	arXiv:1207.5011.
\bibitem{Li08}  H-Z. Li. \emph{On the lower bound of the K-energy and F-functional. }Osaka J. Math. 45 (2008), no. 1, 253-264. 
\bibitem{Lu67} Y-C. Lu, \emph{Holomorphic mappings of complex manifolds}.  J. Diff. Geometry, 2(1968), 291-312.
\bibitem{SW12} J. Song, X-W. Wang. \emph{The greatest Ricci lower bound, conical Einstein metrics and the Chern number inequality.}   arXiv:1207.4839.
\bibitem{G} G. Szekelyhidi. {\it  Greatest lower bounds on the Ricci curvature of Fano manifolds.} Compos. Math. 147 (2011), no. 1, 319-331.
\bibitem{Yau78} S-T. Yau. \emph{On the {R}icci curvature of a compact {K}\"ahler manifold and the
  complex {M}onge-{A}mpere equation, ${I}^*$. }  Comm. Pure Appl. Math. 31. 339-441, 1978.
\end{thebibliography}
\end{document}